\documentclass[letterpaper, 10 pt, conference]{ieeeconf}
\usepackage{amsmath,amssymb,euscript,psfrag,latexsym,graphicx}
\usepackage{bbm,color,amstext,wasysym,subfig,parskip,balance}
\graphicspath{{./},{./figures/}}

\newtheorem{thm}{Theorem}

\newcommand{\mR}{{\mathbb R}}

\newcommand{\cH}{{\mathcal H}}
\newcommand{\cD}{{\mathcal D}}

\newcommand{\cS}{{\mathcal S}}

\newcommand{\trace}{\operatorname{tr}}

\definecolor{grey}{rgb}{0.6,0.6,0.6}
\definecolor{lightgray}{rgb}{0.97,.99,0.99}

\begin{document}
\title{Matricial Wasserstein-1 Distance}

\author{Yongxin Chen, Tryphon T. Georgiou, Lipeng Ning,  and Allen Tannenbaum
\thanks{Y.\ Chen is with the Department of Medical Physics, Memorial Sloan Kettering Cancer Center, NY; email: chen2468@umn.edu}
\thanks{T.\ T. Georgiou is with the Department of Mechanical and Aerospace Engineering, University of California, Irvine, CA; email: tryphon@uci.edu}
\thanks{L.\ Ning is with Brigham and Women's Hospital (Harvard Medical School), MA.; email: lning@bwh.harvard.edu}
\thanks{A.\ Tannenbaum is with the Departments of Computer Science and Applied Mathematics \& Statistics, Stony Brook University, NY; email: allen.tannenbaum@stonybrook.edu}}

\maketitle

\begin{abstract}
We propose an extension of the Wasserstein 1-metric ($W_1$) for matrix probability densities, matrix-valued density measures, and an unbalanced interpretation of mass transport. We use duality theory and, in particular, a ``dual of the dual'' formulation of $W_1$. This matrix analogue of the Earth Mover's Distance has several attractive features including ease of computation.
\end{abstract}

\section{Introduction}

Optimal mass transport (OMT) has proven to be a powerful methodology for numerous problems in physics, probability, information theory, fluid mechanics, econometrics, systems and control, computer vision, and signal/image processing \cite{Kantorovich1948,Rachev,French,Villani,TGT,Mueller}.
Developments along purely controls-related issues ensued when it was recognized that mass transport may be naturally reformulated as a stochastic control problem; see \cite{PW,MT,Leonard,DP,CGP,CGP0,CGP1,CGP2,CGPT} and the references therein.

Historically, the problem of OMT \cite{Rachev,Villani} began with the question of minimizing the effort of transporting one distribution to another, typically with a {\em cost proportional to the Euclidean distance} between starting and ending points of the mass being transported. However, the control-theoretic reformulation \cite{French} which was at the root of the aforementioned developments was based on the {\em choice of a quadratic cost}. The quadratic cost allowed the interpretation of the transport effort as an action integral and gave rise to a Riemannian structure on the space of distributions \cite{McC97,Jordan,Otto}. The originality in our present work is two-fold. First, we formulate the  transport problem with an $L_1$ cost in a similar manner, as a control problem with an $L_1$-path cost functional, and secondly,  we develop theory for shaping flows of matrix-valued distributions which is a non-trivial generalization of classical OMT.

The relevance of OMT on flows of matrix-valued distributions was already recognized in \cite{Lipeng,NinGeo14} and was cast as a control problem as well, albeit in a quadratic-cost setting. At that point, interest in the geometry of matrix-valued distributions stemmed from applications to spectral analysis of vector-valued time series (see \cite{Lipeng} and the references therein). Yet soon it became aparent that flows of matrix-valued distributions represent evolution of quantum systems. In fact, there has been a burst of activity in applying ideas of quantum mechanics to OMT of matrix-valued densities as well as, utilizing an OMT framework to study the dynamics of quantum systems: three groups \cite{Carlen, Chen, Mielke} independently and simultaneously developed quantum mechanical frameworks for defining a Wasserstein-2 distance on matrix-valued densities (normalized to have trace 1), via a variational formalism generalizing the  work of \cite{French}. We note that \cite{Carlen, Chen, Mielke} develop matrix-valued generalizations of the Wasserstein 2-metric ($W_2$) and explore the {\em Riemannian-like} structure for studying the entropic flows of quantum states.

Thus, in our present note, we develop a natural extension of the Wasserstein 1-metric to matrix-valued densities and matrix-valued measures.
Our point of view is somewhat different from the earlier works on matricial Wassserstein-2 metrics. We mainly use duality theory \cite{Evans, Villani}. Further, we do not employ the Benamou and Brenier \cite{French} control formulation of OMT, but rather the Kantorovich-Rubinstein duality. This new scheme is computationally more attractive and, moreover, it is especially appealing when specialized to weighted graphs (discrete spaces) that are sparse (few edges), as is the case for many real-world networks \cite{Sandhu,Sandhu1,jonck}.

The present paper is structured as follows. Section \ref{sec:omtscalar} is a quick review of several different formulations of Wasserstein-1 distance in the scalar setting. Using the quantum gradient operator defined in Section \ref{sec:quantumgrad}, we generalize the Wasserstein-1 metric to the space of density matrices in Section \ref{sec:omtmatrix}. The case where the two marginal matrices have different traces is discussed in Section \ref{sec:unbalanced}. We finally extend the framework to deal with matrix-valued densities in Section \ref{sec:space}, which may find applications in multivariate spectra analysis as well as comparing stable multi-inputs multi-outputs (MIMO) systems. The paper concludes with an academic example in Section \ref{sec:example}.

\section{Optimal mass transport}\label{sec:omtscalar}
We begin with duality theory, explained for scalar densities, upon which our matricial generalization of the Wasserstein-1 metric is based.

Given two probability densities $\rho_0$ and $\rho_1$ on $\mR^m$, the Wasserstein-1 distance between them is
	\begin{align}\label{eq:scalarprimal}
		W_1 (\rho_0,\rho_1):=\inf_{\pi\in \Pi(\rho_0,\rho_1)} \int_{\mR^m\times\mR^m} \|x-y\| \pi(dx,dy),
	\end{align}
where $\Pi(\rho_0,\rho_1)$ denotes the set of couplings between $\rho_0$ and $\rho_1$. The Wasserstein-1 distance has a dual formulation via the following result due to Kantorovich and Rubinstein \cite{Evans, Rachev, Villani}:
	\begin{align}\label{eq:scalardual}
		W_1 (\rho_0,\rho_1)&=\sup_f \left\{\int_{\mR^m} f(x)(\rho_0(x)-\rho_1(x))dx\right.\\
		&\hspace*{2cm}\left.\phantom{\int_{\mR^m}}\mbox{subject to }\|f\|_{\rm Lip} \le 1\right\},\nonumber
	\end{align}
	where $\|f\|_{\rm Lip}$ denotes the Lipschitz constant.
When $f$ is differentiable, $\|f\|_{\rm Lip} =\|\nabla_x f\|$. It follows that,
	\begin{align}\label{eq:scalardual1}
		W_1 (\rho_0,\rho_1)&=\sup_f \left\{\int_{\mR^m} f(x)(\rho_0(x)-\rho_1(x))dx\right.\\
		&\hspace*{2cm}\left.\phantom{\int_{\mR^m}}\mbox{subject to }\|\nabla_x f\| \le 1\right\},\nonumber
	\end{align}	
Starting from \eqref{eq:scalardual1}, by once again considering the dual, we readily obtain the very important reformulation
	\begin{align}\label{eq:scalarprimal1}
		W_1 (\rho_0,\rho_1)&=\inf_{u(\cdot)} \left\{\int_{\mR^m} \|u(x)\| dx \right.\\
		&\hspace*{1cm}\left.\phantom{\int_{\mR^m}}\mbox{subject to }\rho_0-\rho_1 +\nabla_x\cdot u=0 \right\},\nonumber
	\end{align}
where the (Lagrange) optimization variable $u$ now represents flux. Alternatively, this can be written as the control-optimization problem in the Benamou-Brenier style \cite{French}
	\begin{align}\label{eq:scalarprimal1control}
		W_1 (\rho_0,\rho_1)&=\inf_{u(\cdot,\cdot)} \left\{\int_0^1\int_{\mR^m} \|u(t,x)\| dx dt\right.\\
		&\phantom{x\int_{\mR^m}}\mbox{subject to } \frac{\partial \rho(t,x)}{\partial t}+\nabla_x\cdot u(t,x)=0,\nonumber\\
		&\left.\phantom{\int_{\mR^m}}\;\mbox{and }\rho(0,x)=\rho_0(x),\;\rho_1(1,x)=\rho_1(x) \right\},\nonumber
	\end{align}

This ``dual of the dual'' formulation turns the Kantorovich and Rubinstein into a control problem to determine a suitable velocity (control vector) $u$. We remark that from a computational standpoint, when applied to discrete spaces (graphs),
this formulation leads to a very substantial computational benefit in the case of sparse graphs; this is due to the fact that \eqref{eq:scalarprimal}
involves solving systems of the order of the square of the number of nodes, while equation~\eqref{eq:scalarprimal1}, solving systems of the order of the number of edges.

\section{Gradient on space of Hermitian matrices}\label{sec:quantumgrad}

We closely follow the treatment in \cite{Chen}. In particular, we will need a notion of gradient on the space of Hermitian matrices and its dual, i.e. the divergence.

Denote by $\cH$ and $\cS$ the set of  $n\times n$ Hermitian and skew-Hermitian matrices, respectively. We will assume that all of our matrices are of fixed size $n\times n$. Next, we denote the space of block-column vectors consisting of $N$ elements in $\cS$ and $\cH$ as $\cS^N$ and $\cH^N$, respectively. We also let $\cH_+$ and $\cH_{++}$ denote the cones of nonnegative and positive-definite matrices, respectively, and
	\begin{eqnarray}\label{eq:D}
    \cD &:=&\{\rho \in \cH_{+} \mid \trace(\rho)=1\},\\
	\cD_+ &:=&\{\rho \in \cH_{++} \mid \trace(\rho)=1\}.
    \end{eqnarray}
We note that the tangent space of $\cD_+,$ at any $\rho\in \cD+$ is given by
	\begin{equation}\label{eq:Trho}
		T_\rho=\{ \delta \in \cH \mid \trace(\delta)=0\},
	\end{equation}
and we use the standard notion of inner product, namely
	\[
		\langle X,Y\rangle=\trace(X^*Y),
	\]
for both $\cH$ and $\cS$.
For $X, Y\in \cH^N$ ($\cS^N$),
	\[
		\langle X, Y\rangle=\sum_{k=1}^N \trace(X_k^*Y_k).
	\]
Given $X=[X_1^*,\cdots,X_N^*]^* \in \cH^N$ ($\cS^N$), $Y\in \cH$ ($\cS$), set
	\[
		XY=\left[\begin{array}{c}
		X_1\\
		\vdots \\
		X_N
		\end{array}\right]Y
		:=
		\left[\begin{array}{c}
		X_1Y\\
		\vdots \\
		X_N Y
		\end{array}\right],
	\]
and
	\[
		YX=Y\left[\begin{array}{c}
		X_1\\
		\vdots \\
		X_N
		\end{array}\right]
		:=
		\left[\begin{array}{c}
		YX_1\\
		\vdots \\
		YX_N
		\end{array}\right].
	\]

For a given $L\in\cH^N$ we define
	\begin{equation}\label{eq:gradient}
		\nabla_L: \cH \rightarrow {\cS}^N, ~~X \mapsto
		\left[ \begin{array}{c}
		L_1 X-XL_1\\
		\vdots \\
		L_N X-X L_N
		\end{array}\right]
	\end{equation}
to be the \textbf{\emph{gradient operator}}.
By analogy with the ordinary multivariable calculus, we refer to its dual with respect to the Hilbert-Schmidt inner product as the (negative) \textbf{\emph{divergence operator}}, and this is
	\begin{equation}\label{eq:divergence}
		\nabla_L^*: {\cS}^N \rightarrow \cH,~~Y=
		\left[ \begin{array}{c}
		Y_1\\
		\vdots \\
		Y_N
		\end{array}\right]
		\mapsto
		\sum_k^N L_k Y_k-Y_k L_k,
	\end{equation}
i.e., $\nabla_L^*$ is defined by means of the identity
	\[
		\langle \nabla_L X , Y\rangle =\langle X , \nabla_L^* Y\rangle.
	\]
A standing assumption throughout, is that the null space of $\nabla_L$, denoted by ${\rm ker}(\nabla_L)$, contains only scalar multiples of the identity matrix.

\section{Wassertein-1 distance for density matrices}\label{sec:omtmatrix}

In this section, we show that both \eqref{eq:scalardual1} and \eqref{eq:scalarprimal1} have natural counterparts for probability density matrices, i.e. matrices in $\cD$.
This set-up obviously works for matrices in $\cH_+$ of equal trace.

We treat \eqref{eq:scalardual1} as our starting definition and define the $W_1$ distance in the space of density matrices as
	\begin{equation}\label{eq:W1}
		W_1 (\rho_0,\rho_1):=\sup_{f\in \cH} \left\{\trace[f(\rho_0-\rho_1)]~\mid~\|\nabla_L f\| \le 1\right\}.
	\end{equation}
Here $\|\cdot\|$ is the operator norm. The above is well-defined since by assumption, the null space of $\nabla_L$ is spanned by the identity matrix $I$. As above, we have that
	\[
		\nabla_L f=
		\left[\begin{array}{c}
		L_1f-fL_1\\
		\vdots \\
		L_N f-fL_N
		\end{array}\right].
	\]
This should be compared to the Connes spectral distance \cite{connes}, which is given by
	\[
		d_D(\rho_0,\rho_1)=\sup_{f\in \cH}\left\{\trace[f(\rho_0-\rho_1)]~\mid~\|[D,f]\| \le 1\right\}.
	\]
	
It is not difficult to see that the dual of \eqref{eq:W1} is
	\begin{equation}\label{eq:W1dual}
		\hat W_1 (\rho_0,\rho_1)=\inf_{u\in\cS^N}\left\{\|u\|_*~\mid~\rho_0-\rho_1-\nabla_L^* u=0\right\},
	\end{equation}
which is the counterpart of \eqref{eq:scalarprimal1}. Here $\|\cdot\|_*$ denotes the nuclear norm \cite{nuclear}.
In particular, we have the following theorems.
%The proofs of the following theorems are straightforward and will be given in the full version of this paper.
\begin{thm} Notation as above. Then
	\[
	W_1 (\rho_0,\rho_1)=\hat W_1 (\rho_0,\rho_1).
	\]
\end{thm}

\begin{proof}
We start from \eqref{eq:W1dual} and use the fact that
	\[
		\|u\|_*=\sup_{g\in\cS^N, \|g\|\le 1} \langle u, g\rangle.
	\]
It follows
	\begin{eqnarray*}
	\hat W_1 (\rho_0,\rho_1)&=&\inf_{u}\sup_{f}\left\{\|u\|_*
	+\langle f, \rho_0-\rho_1-\nabla_L^* u\rangle\right\}
	\\
	&=&\inf_{u}\sup_{f,\|g\|\le 1}\left\{\langle u, g\rangle
	+\langle f, \rho_0-\rho_1-\nabla_L^* u\rangle\right\}
	\\
	&=&\inf_{u}\sup_{f,\|g\|\le 1}\left\{\langle u, g-\nabla_Lf\rangle
	+\langle f, \rho_0-\rho_1\rangle\right\}
	\\
	&\ge& \sup_{f,\|g\|\le 1}\inf_{u}\left\{\langle u, g-\nabla_Lf\rangle
	+\langle f, \rho_0-\rho_1\rangle\right\}
	\\
	&=& \sup_{f,\|g\|\le 1}\left\{\langle f, \rho_0-\rho_1\rangle~\mid~
	g=\nabla_Lf\right\}
	\\
	&=& \sup_{f}\left\{\langle f, \rho_0-\rho_1\rangle~\mid~
	\|\nabla_L f\|\le 1\right\}\\
    &=& W_1(\rho_0,\rho_1).
	\end{eqnarray*}
This implies that \eqref{eq:W1} and \eqref{eq:W1dual} are dual to each other. Since both of them are strictly feasible, the duality gap is zero. Therefore $W_1 (\rho_0,\rho_1)=\hat W_1 (\rho_0,\rho_1)$.
\end{proof}

\begin{thm}\label{thm:W1metric}
The $W_1$ distance defined as in \eqref{eq:W1} is a metric on the space of density matrices $\cD$.
\end{thm}
\begin{proof}
Obviously $W_1(\rho_0,\rho_1)\ge 0$ holds with equality if and only if $\rho_0=\rho_1$. The symmetric property that $W_1(\rho_0,\rho_1)=W_1(\rho_1,\rho_0)$ is also clear from the definition. Here we prove the triangle inequality. That is, for any $\rho_0, \rho_1, \rho_2\in\cD)$, we have
	\[
		W_1(\rho_0,\rho_2)\le W_1(\rho_0,\rho_1)+W_1(\rho_1,\rho_2).
	\]
It is easier to see this from the dual formulation \eqref{eq:W1dual}. Let $u_1$, $u_2$ be the optimal fluxes for $(\rho_0,\rho_1)$ and $(\rho_1,\rho_2)$ respectively. Then $u_1+u_2$ is a feasible flux for $(\rho_0,\rho_2)$, namely,
	\[
		\rho_0-\rho_2-\nabla_L^* (u_1+u_2)=0.
	\]
It follows that
	\[
		W_1(\rho_0,\rho_2)\le \|u_1+u_2\|_*\le\|u_1\|_*+\|u_2\|_*,
	\]
which completes the proof.
\end{proof}

\section{Wassertein-1 distance: the unbalanced case}\label{sec:unbalanced}
In this section, we extend the definition of Wasserstein-1 distance to the space nonnegative matrices $\cH_+,$ i.e., we remove the constraint of both matrices having equal traces. Compare also with some very interesting recent  work \cite{osher} on fast computational methods for $W_1$ in the unbalanced scalar case.

In order to compare matrices of unequal trace we relax the constraint in \eqref{eq:W1dual}, which forces
	$\trace(\rho_0)=\trace(\rho_1)$,
by introducing a ``source'' term $v\in \cH$. That is, we replace our continuity equation \eqref{eq:W1dual} with
	\begin{equation}
		\rho_0-\rho_1-\nabla_L^* u-v=0.
	\end{equation}
With this added source, we define a Wasserstein-1 distance in $\cH_+$ as follows.   Given $\rho_0, \rho_1\in \cH_+$, we define
	\begin{align}\label{eq:W1dual1}
		&V_1 (\rho_0,\rho_1)\!=\!\inf_{\substack{u\in\cS^N\\v\in\cH}}\left\{\|u\|_*+\alpha \|v\|_*\!\mid\!\rho_0\!-\!\rho_1\!-\!\nabla_L^* u-v=0\right\}.
	\end{align}
Here $\alpha>0$ measures the relative significance between $u$ and $v$.

Another natural way to compare $\rho_0,\rho_1\in\cH_+$ is by finding $\mu, \nu \in \cH_+$ having equal trace that are close to $\rho_0, \rho_1$ in some norm (here taken to be the nuclear norm), as well as close to one another. More specifically, we seek $\mu, \nu$ to minimize
	\begin{equation}
		W_1(\mu,\nu)+\alpha\|\rho_0-\mu\|_*+\alpha\|\rho_1-\nu\|_*.
	\end{equation}
Putting the two terms together we obtain the following definition of Wasserstein-1 distance
	\begin{subequations}\label{eq:V1dual}
	\begin{eqnarray}
	\!\!\hat V_1(\rho_0,\rho_1)\!\!\!\!&=&\!\!\!\!\!\!\inf_{\substack{u\in\cS^N \\ \mu, \nu \in \cH_+}}\!\!\|u\|_*\!+\!\alpha\|\rho_0\!-\!\mu\|_*\!+\!
	\alpha\|\rho_1-\nu\|_*
	\\&& \mu-\nu-\nabla_L^* u=0,
	\\&&\trace(\mu)=\trace(\nu).
	\end{eqnarray}
	\end{subequations}
It turns out these two relaxations of $W_1$ are in fact equivalent.

%As above, the proofs of the following results are straightforward and will be given in the full version of this paper.
\begin{thm} With notation and assumptions as above,
	\begin{equation}
		V_1(\rho_0,\rho_1)= \hat V_1(\rho_0,\rho_1).
	\end{equation}
\end{thm}
\begin{proof}
Clearly, $\hat V_1(\rho_0,\rho_1) \ge V_1(\rho_0,\rho_1)$. On the other hand, let $u, v$ be a minimizer of \eqref{eq:W1dual1}, and $v=v_1-v_0$ with $v_0,v_1\in \cH_+$, i.e., $v_0,v_1$ are the negative and positive parts of $v$ respectively, then $\mu=\rho_0+v_0, \nu=\rho_1+v_1$ together with $u$ is a feasible solution to \eqref{eq:V1dual}. With this solution,
	\begin{eqnarray*}
		\hat V_1(\rho_0,\rho_1) &\le&\|u\|_*+\alpha\|\rho_0-\mu\|_*+\alpha\|\rho_1-\nu\|_*
		\\&=& \|u\|_*+\alpha\|v_0\|_*+\alpha\|v_1\|_*
		\\&=& \|u\|_*+\alpha\|v\|_*,
	\end{eqnarray*}
which implies that $\hat V_1(\rho_0,\rho_1)\le V_1(\rho_0,\rho_1)$. This completes the proof.
\end{proof}

\begin{thm}
The formula \eqref{eq:W1dual1} defines a metric on $\cH_+$.
\end{thm}
\begin{proof}
The proof follows exactly the same lines as in Theorem \ref{thm:W1metric}.
\end{proof}

Using the technique of Lagrangian multipliers one can deduce the dual formulation of \eqref{eq:W1dual} and establish the following:
\begin{thm} Notation as above. Then
	\begin{equation}
		V_1(\rho_0,\rho_1)=\sup_{f\in\cH}\left\{\trace[f(\rho_0-\rho_1)]~\mid~\|\nabla_L f\| \le 1, ~ \|f\| \le \alpha\right\}.
	\end{equation}
\end{thm}
\begin{proof}
Straight calculation gives
	\begin{eqnarray*}
	V_1 (\rho_0,\rho_1)&=&\inf_{u,v}\sup_f\{\|u\|_*+\alpha \|v\|_*+
	\\&&\langle f, \rho_0-\rho_1-\nabla_L^* u-v\rangle\}
	\\&=&\inf_{u,v}\sup_{f,\|g\|\le 1, \|h\|\le 1}\{\langle u, g\rangle+\alpha \langle v,h\rangle
	\\&&+
	\langle f, \rho_0-\rho_1-\nabla_L^* u-v\rangle\}
	\\&&\hspace*{-1cm}\ge\sup_{f,\|g\|\le 1, \|h\|\le 1}\inf_{u,v}\{\langle u, g-\nabla_L f\rangle
	+\langle v,\alpha h-f\rangle
	\\&&+
	\langle f, \rho_0-\rho_1\rangle\}
	\\&&\hspace*{-1cm}=\sup_f \left\{\langle f, \rho_0-\rho_1\rangle ~\mid ~
	\|\nabla_L f\| \le 1, ~ \|f\| \le \alpha\right\}.
	\end{eqnarray*}
This together with the strong duality completes the proof.
\end{proof}

\section{Wasserstein-1 distance for matrix-valued densities}\label{sec:space}
With little effort we are able to generalize the definition of Wasserstein-1 distance to the space of matrix-valued densities. Examples of matrix-valued densities include power spectra of multivariate time series, stress tensors, diffusion tensors and so on, and hence our motivation in considering matrix-valued distribution on possibly more than a one dimensional spatial coordinates.

Given two matrix-valued densities $\rho_0, \rho_1$ satisfying
	\begin{equation}\label{eq:balanced}
		\int_{\mR^m} \trace(\rho_0(x))dx=\int_{\mR^m} \trace(\rho_1(x))dx,
	\end{equation}
 we can define their Wasserstein-1 distance as
	\begin{eqnarray*}
		W_1 (\rho_0,\rho_1)&:=&\sup_{f\in \cH} \left\{\int_{\mR^m}\trace[f(x)(\rho_0(x)-\rho_1(x))]dx~\mid~\right.
		\\&&\left.\left\|\left[\begin{matrix}\nabla_x f\\\nabla_L f\end{matrix}\right]\right\| \le 1\right\},
	\end{eqnarray*}
or through its dual
	\begin{eqnarray*}
		W_1 (\rho_0,\rho_1)&=&\inf_{\substack{u_1\in\cH^m \\ u_2\in\cS^N}} \left\{\int_{\mR^m}
		\left\|\left[\begin{matrix}u_1(x)\\u_2(x)\end{matrix}\right]\right\|_*dx~\mid~\right.
		\\&&\left.
		\rho_0-\rho_1+\nabla_x\cdot u_1-\nabla_L^* u_2=0\right\}.
	\end{eqnarray*}

For more general densities where condition \eqref{eq:balanced} may not be valid, we define
	 \begin{eqnarray*}
		V_1 (\rho_0,\rho_1)&:=&\sup_{f\in \cH} \left\{\int_{\mR^m}\trace[f(x)(\rho_0(x)-\rho_1(x))]dx~\mid~
		\right. \\&&\left.
		\left\|\left[\begin{matrix}\nabla_x f\\\nabla_L f\end{matrix}\right]\right\| \le 1, \|f\|\le \alpha\right\},
	\end{eqnarray*}
or, equivalently,
	\begin{eqnarray*}
		V_1 (\rho_0,\rho_1)&=&\inf_{\substack{u_1\in\cH^m \\ u_2\in\cS^N, v\in \cH}} \left\{\int_{\mR^m}
		\left\|\left[\begin{matrix}u_1(x)\\u_2(x)\end{matrix}\right]\right\|_*+\alpha\|v\|_*dx~\mid~
		\right.\\&&\left.
		\rho_0-\rho_1+\nabla_x\cdot u_1-\nabla_L^* u_2-v=0\right\}.
	\end{eqnarray*}

One can introduce positive coefficients $\beta_1>0, \beta_2>0$ to  trade-off the relative importance of $u_1$ and $u_2$ in establishing correspondence between the two distributions as follows:
	 \begin{align*}
		V_1 (\rho_0,\rho_1)&:=\sup_{f\in \cH} \left\{\int_{\mR^m}\trace[f(x)(\rho_0(x)-\rho_1(x))]dx~\mid~
		\right.\\&\left.
		\left\|\left[\begin{matrix}\beta_1\nabla_x f\\\beta_2\nabla_L f\end{matrix}\right]\right\| \le 1, \|f\|\le \alpha\right\},
	\end{align*}
or, equivalently,
	\begin{align*}
		V_1 (\rho_0,\rho_1)&=\inf_{\substack{u_1\in\cH^m \\ u_2\in\cS^N, v\in \cH}} \left\{\int_{\mR^m}
		\left\|\left[\begin{matrix}u_1(x)\\u_2(x)\end{matrix}\right)\right\|_*+\alpha\|v\|_*dx~\mid~
		\right.\\&\left.
		\rho_0-\rho_1+\beta_1\nabla_x\cdot u_1-\beta_2\nabla_L^* u_2-v=0\right\}.
	\end{align*}
	
\section{Example}\label{sec:example}
We use our framework to compare power spectra of multivariate time series (in discrete time). Evidently, the distance between two power spectra induces a distance between corresponding linear modeling filters and, thereby, can be used to compare (stable) MIMO systems \cite{Lipeng}.

Consider the three power spectra as shown in Figure~\ref{fig:spectra} (in different colors). What is shown in the three subplots are power spectra of two time series (in subplots (a) and (c)) and their cross-spectrum (in subplot (b)) as functions of time (the phase of the cross spectra are not shown). Thus, the three different colors represent the three different matrix-valued power spectra given by:
	\begin{align*}
		\rho_0(\theta) &= \left[\begin{matrix}1 & 0.4 \\0 & 1\end{matrix}\right]
		\left[\begin{matrix}0.01 & 0 \\0 & \frac{0.7}{|a_0(e^{j\theta})|^2}\end{matrix}\right]
		\left[\begin{matrix}1 & 0 \\0.4 & 1\end{matrix}\right]
		\\
		\rho_1(\theta) &={\footnotesize \left[\begin{matrix}1 & 0.5 \\0.5e^{j\theta} & 1\end{matrix}\right]}
		\left[\begin{matrix}\frac{0.5}{|a_1(e^{j\theta})|^2} & 0 \\ 0 & \frac{0.5}{|a_1(e^{j\theta})|^2}\end{matrix}\right]
		{\footnotesize \left[\begin{matrix}1 & 0.5e^{-j\theta} \\0.5 & 1\end{matrix}\right]}
		\\
		\rho_2(\theta) &= \left[\begin{matrix}1 & 0 \\0.4e^{j\theta} & 1\end{matrix}\right]
		\left[\begin{matrix}\frac{2}{|a_2(e^{j\theta})|^2}& 0 \\0 & 0.02 \end{matrix}\right]
		\left[\begin{matrix}1 & 0.4e^{-j\theta} \\ 0 & 1\end{matrix}\right]
	\end{align*}
where
	\begin{align*}
	a_0(z)&=(1\!-\!1.9\cos(\frac{\pi}{6})z\!-\!0.95^2z^2)\\&\times(1\!-\!1.5\cos(\frac{\pi}{3})z\!+\!0.75^2z^2)
	\\
	a_1(z)&=(1\!\!-\!\!1.9\cos(\frac{2\pi}{3})z\!\!-\!\!0.95^2z^2)\\
	&\times(1\!\!-\!\!1.5\cos(\frac{5\pi}{8})z\!+\!0.75^2z^2)
	\\
	a_2(z)&= (1\!\!-\!\!1.9\cos(\frac{5\pi}{12})z\!\!-\!\!0.95^2z^2)\\
	&\times(1\!\!-\!\!1.5\cos(\frac{\pi}{2})z\!\!+\!\!0.75^2z^2).
	\end{align*}

\begin{figure}[h]
\centering
\subfloat[$\rho(1,1)$]{\includegraphics[width=0.38\textwidth]{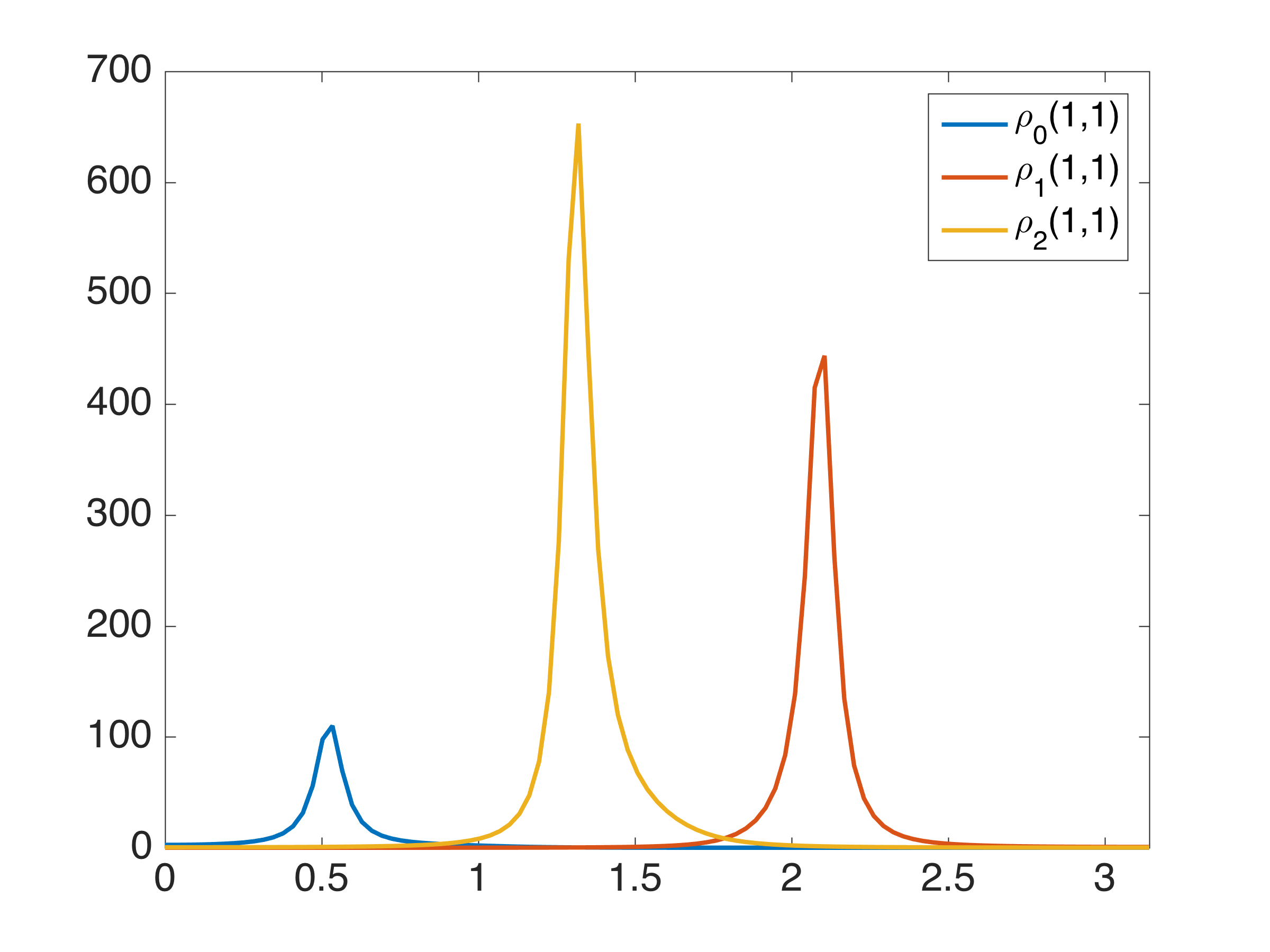}}
\\
\subfloat[$\rho(1,2)$]{\includegraphics[width=0.38\textwidth]{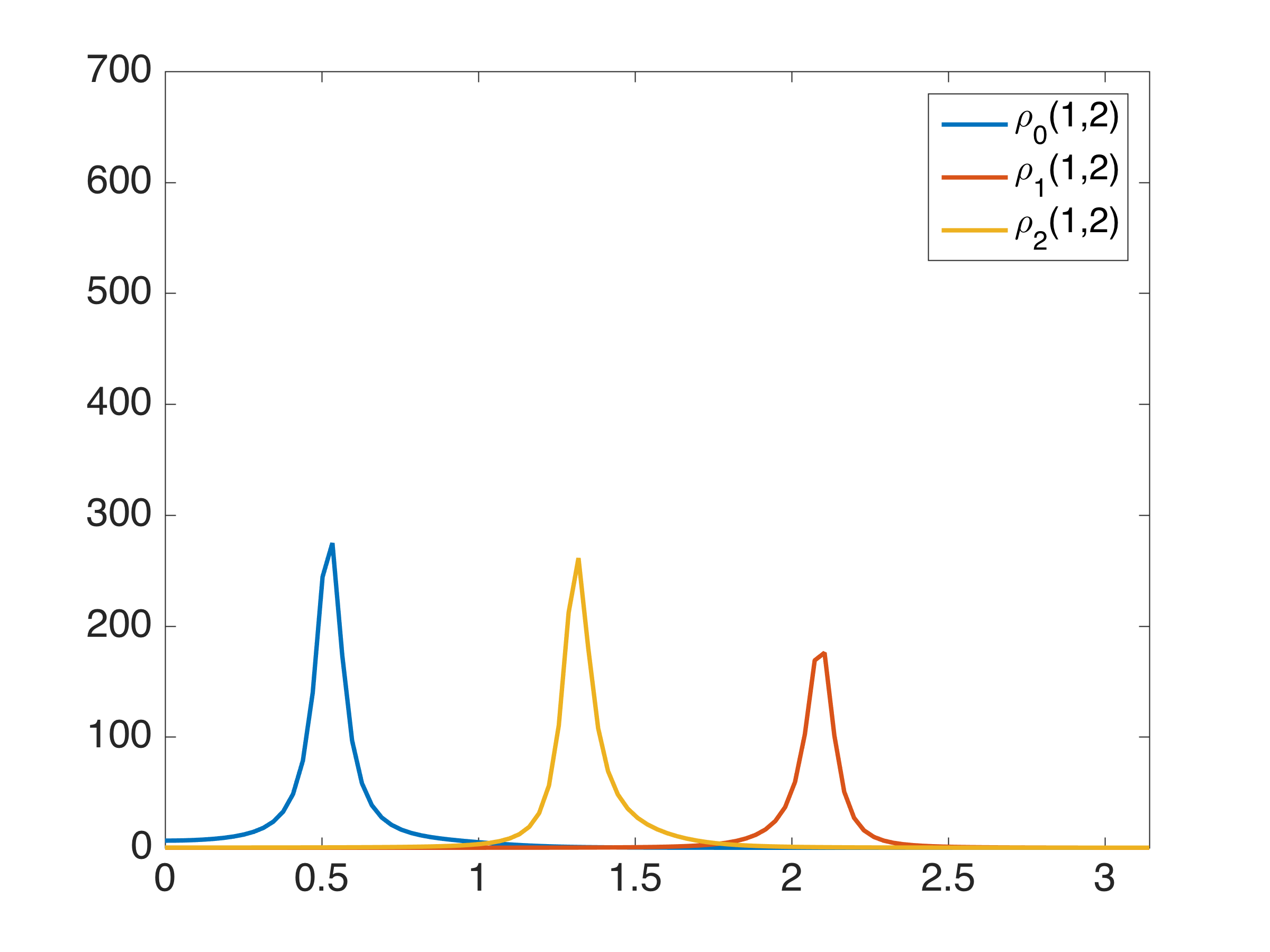}}
\\
\subfloat[$\rho(2,2)$]{\includegraphics[width=0.38\textwidth]{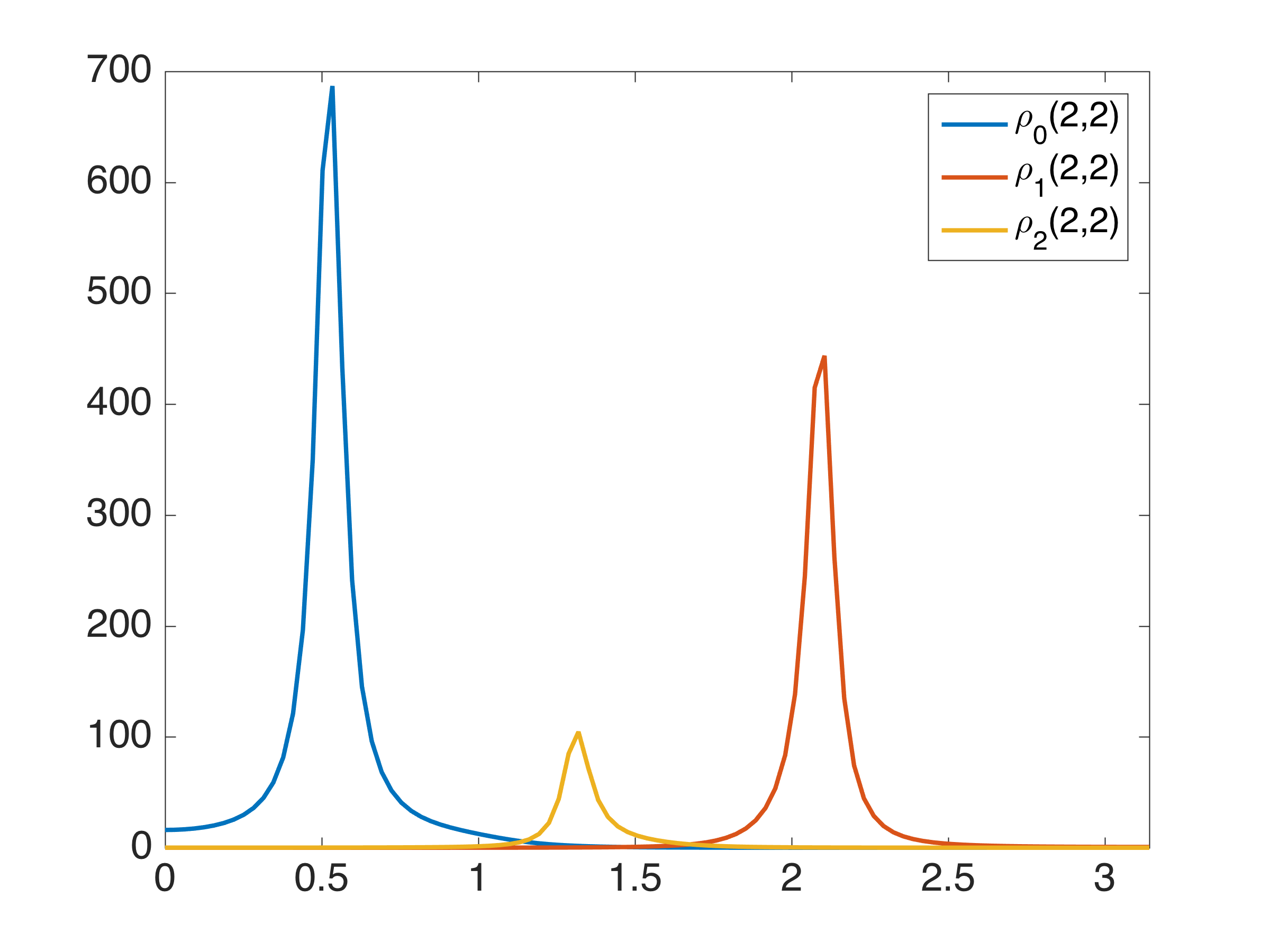}}
    \caption{Power spectra}
    \label{fig:spectra}
\end{figure}

The distances between the each pair for different $\beta_1,\beta_2$ values, $\alpha=1$, and the choice $L=[L_1, \ L_2]$ with
\[L1=\left[\begin{matrix}1& 0\\0 &  0\end{matrix}\right],\;L_2=\left[\begin{matrix}1& 1\\1 &  0\end{matrix}\right],\] are tabulated in Table \ref{table:table1}. We observe that when the penalty on the rotation part is large ($\beta_1>>\beta_2$),  we have $V_1(\rho_0,\rho_2)>V_1(\rho_0,\rho_1)$ and $V_1(\rho_0,\rho_2)>V_1(\rho_2,\rho_1)$. On the other hand, when the penalty on translation is large relative to the cost of rotation ($\beta_1<<\beta_2$), we have $V_1(\rho_0,\rho_1)>V_1(\rho_0,\rho_2)$ and $V_1(\rho_0,\rho_1)>V_1(\rho_1,\rho_2)$. These findings are in agreement with the intuition when observing the relative frequency directionality of power in the three spectra. More specifically, $\rho_1$ requires a significant drift in directionality before we can match it with the other two, while this is less important when comparing $\rho_0$ and $\rho_2$. For this latter case, it is the actual frequency where the power resides that distinguishes the two while the directionality is more in agreement.

What this example underscores is the ability of the metric to be tailored to applications where we need to trade off and compromise, in a principled way, between two vastly different features of matrix-valued distributions, i.e., spatial location versus directionality of the ``intensity.'' What was achieved in this paper is the construction of a suitable and easily computable metric that can be utilized for this purpose.

\begin{table}[t]
\centering
\caption{Distances between power spectra }\label{table:table1}
\begin{tabular}{|c|c|c|c|}
\hline
 & $\rho_0, \rho_1$    & $\rho_1, \rho_2$ & $\rho_0, \rho_2$  \\ \hline
$\beta_1=10, \beta_2=1$       & 77.85    & 77.76 & 137.36    \\ \hline
$\beta_1=1,\beta_2=1$       & 249.40  & 162.03 & 199.78  \\ \hline
$\beta_1=1,\beta_2=10$       & 210.93 & 110.25 & 113.46\\ \hline
\end{tabular}
\end{table}

\section{Future research}

We introduced generalization of the scalar $W_1$ distance to matrices and matrix-valued measures. This new metric, $W_1$, is computationally simpler and more attractive than earlier metrics, based on quadratic cost criteria. In fact, our ``dual of the dual'' formulation makes the metric especially attractive when comparing matrix-valued data on a discrete space (graph, network).

We note that the Wasserstein 1-metric has been used as a tool in defining {\bf\it curvature} \cite{Ollivier} and in analyzing the \textbf{\emph{robustness}} of complex networks derived from scalar-valued data \cite{Sandhu,Sandhu1}. The formalism presented in the current work, suggests alternative notions of curvature and robustness when the nodes of a network carry matrix-valued data, e.g., in diffusion tensor imaging. We plan to pursue such issues in future work.

\section*{Acknowledgements}
This project was supported by AFOSR grants (FA9550-15-1-0045 and FA9550-17-1-0435), grants from the National Center for Research Resources (P41-
RR-013218) and the National Institute of Biomedical Imaging and Bioengineering (P41-EB-015902), National Science Foundation (NSF), and National Institutes of Health (P30-CA-008748 and 1U24CA18092401A1).

\bibliographystyle{plain}

\end{document}